\documentclass[graybox]{svmult}

\usepackage{mathptmx}       
\usepackage{helvet}         
\usepackage{courier}        
\usepackage{type1cm}        
\usepackage{makeidx}         
\usepackage{graphicx}        
\usepackage{multicol}        
\usepackage[bottom]{footmisc}

\usepackage{amsmath,amsfonts,verbatim,dsfont}

\usepackage{soul}
\usepackage{color}

\makeindex             

\newcommand{\nat}{\mathbb{N}}
\newcommand{\real}{\mathbb{R}}
\newcommand{\rn}{{\mathbb{R}^n}}

\newcommand{\Ee}{\mathbb{E}}
\newcommand{\Pp}{\mathbb{P}}
\newcommand{\LL}{\mathbb{L}}
\newcommand{\LLB}{\widehat{\LL}}
\newcommand{\Id}{\mathds 1}

\newcommand{\Cauchy}{\mathsf{C}}

\newcommand{\poisol}{F}

\DeclareMathOperator{\id}{id}
\DeclareMathOperator{\tr}{Tr}
\DeclareMathOperator{\ex}{Ex}

\DeclareMathOperator{\Laplace}{\varDelta}
\DeclareMathOperator{\dist}{\operatorname{dist}}

\begin{document}

\allowdisplaybreaks

\title*{A Probabilistic proof of the breakdown of Besov regularity in $L$-shaped domains}
\author{Victoria  Knopova \and Ren\'e L. Schilling}
\titlerunning{Breakdown of Besov regularity in $L$-shaped domains}
\institute{Victoria  Knopova
\at
Institut f\"ur Mathematische Stochastik,
Fachrichtung Mathematik,
TU Dresden,
01062 Dresden, Germany.
\email{viktoriya.knopova@tu-dresden.de}
\and
Ren\'e L.\ Schilling
\at
Institut f\"ur Mathematische Stochastik,
Fachrichtung Mathematik,
TU Dresden,
01062 Dresden, Germany.
\email{rene.schilling@tu-dresden.de}
}
%
%
\maketitle

\abstract{We provide a probabilistic approach in order to investigate the smoothness of the solution to the Poisson and Dirichlet problems in $L$-shaped domains. In particular, we obtain (probabilistic) integral representations \eqref{fG}, \eqref{fD2}--\eqref{fD22} for the solution. We also recover Grisvard's classic result on the angle-dependent breakdown of the regularity of the solution measured in a Besov scale.
}

\section{Introduction}

Let us consider the (homogeneous) Dirichlet problem
\begin{equation}\label{D1}
\begin{aligned}
\Laplace f &=0 && \text{in $G$},\\
f|_{\partial G}& = h &&\text{on $\partial G$},
\end{aligned}
\end{equation}
where $G\subset \mathbb{R}^d$ is a domain with Lipschitz  boundary $\partial G$ and $\Laplace$ denotes the Laplace operator, i.e.\ $\Laplace= \sum_{i=1}^d \frac{\partial^2}{\partial x_i^2}$.   In order to show that there exists a solution to \eqref{D1} which belongs to some subspace of $L_p(G)$, say, to the Besov space $B_{pp}^\sigma(G)$, $\sigma>0$, it is necessary that  $h$  is an element of the trace space of $B_{pp}^\sigma(G)$ on $\partial G$; it is well known that the trace space is given by $B_{pp}^{\sigma-1/p}(\partial G)$, see Jerison \& Kenig \cite[Theorem 3.1]{JK95}, a more general version can be found in Jonsson \& Wallin
\cite[Chapter VII]{JW84}, and for domains with $C^\infty$-boundary a good reference is Triebel \cite[Sections 3.3.3--4]{Tr83}. The smoothness of the solution $f$, expressed by the parameter $\sigma$ in $B_{pp}^\sigma(G)$, is, however, not only determined by the smoothness of $h$, but also by the geometry of $G$. It seems that Grisvard \cite{Gr85} is the first author to quantify this in the case when $G$ is a non-convex polygon. Subsequently, partly due to its relevance in scientific computing, this problem attracted a lot of attention; for instance, it  was  studied by Jerison \& Kenig \cite{JK95}, by  Dahlke \& DeVore \cite{DD97} in connection with wavelet representations of Besov functions, by Mitrea \& Mitrea \cite{MM03} and Mitrea, Mitrea \& Yan  \cite{MMY10} in H\"older spaces,  to mention but a few references.

In this note we use a probabilistic approach to the problem and we obtain a probabilistic interpretation in the special case when $G$ is an \emph{$L$-shaped domain} of the form $\LL:=  \mathbb{R}^2\backslash \{(x,y):\,x,y\geq 0\}$, see Figure~\ref{L-shape1}, and in an $L_2$-setting.
\begin{figure}\label{L-shape1}
\sidecaption
\includegraphics[scale=0.8]{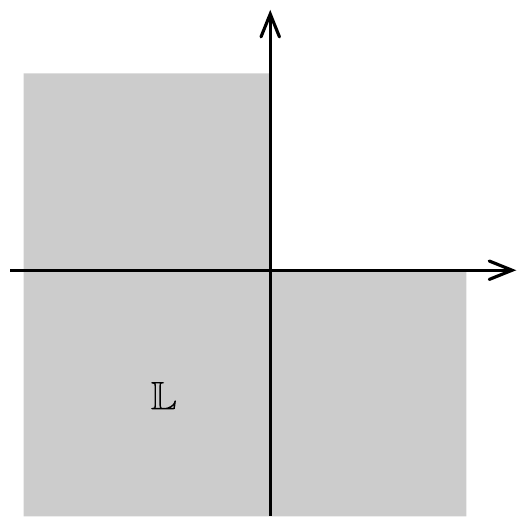}
\caption{The $L$-shaped model domain $\LL\subset\real^2$.}
\end{figure}
This is the model problem for all non-convex domains with an obtuse interior angle.
In this case the Besov space $B_{22}^\sigma(\LL)$ coincides with the Sobolev--Slobodetskij space $W_2^\sigma(\LL)$.  In particular, we
\begin{itemize}
\item
    give a probabilistic interpretation of the solution to \eqref{D1} with $G=\LL$;

\item
    provide a different proof of the fact that the critical order of smoothness of $f$ is $\sigma< \pi/\frac{3\pi}{2}=\frac{2}{3}$, i.e.´\ even for $h\in C_0^2(\partial \LL)$ we may have
    \begin{equation}\label{eq-main}
        f\in W_{2,\mathrm{loc}}^{1+ \sigma}(\LL), \quad \sigma<  \tfrac{2}{3},
        \qquad\text{and}\qquad
        f\notin W_2^{1+ \sigma}(\LL), \quad \sigma\geq \tfrac{2}{3};
    \end{equation}

\item
    apply the ``breakdown of regularity'' result to the Poisson (or inhomogeneous Dirichlet) problem.
\end{itemize}
It is clear that this result holds in a more general setting, if we replace the obtuse angle $3\pi/2$ by some $\theta\in (\pi, 2\pi)$.

Results of this type were proved for polygons and in a H\"older space setting by Mitrea \& Mitrea \cite{MM03}. Technically, our  proof is close (but different) to that given in \cite{MM03}---yet our staring idea is different. Dahlke \& DeVore \cite{DD97} proved this regularity result analytically using a wavelet basis for $L_p$-Besov spaces.

Problem \eqref{D1} is closely related to the Poisson (or nonhomogeneous Dirichlet) problem
\begin{equation}\label{P1}
\begin{aligned}
    \Laplace\poisol  &=g && \text{on $G$},\\
    \poisol|_{\partial G}& = 0 &&\text{on $\partial G$}.
    \end{aligned}
\end{equation}
If $G$ is bounded and  has a $C^\infty$-boundary, the problems \eqref{D1} and \eqref{P1} are equivalent. Indeed, in this case for every right-hand side $g\in L_2(G)$ of \eqref{P1} there exists a unique solution $\poisol \in W_2^2(G)$,  see Triebel \cite[Theorem 4.3.3]{Tr83}. Denote by $N$ the Newtonian potential  on $\mathbb{R}^d$ and define $w:=g* N$; clearly, $\Laplace w=g$ on $G$ and $w\in W_2^2(G)$. Since the boundary is smooth, there is a continuous  linear trace operator $\tr: W_2^{2}(G)\to W_p^{3/2}(\partial G)$ as well as a continuous linear extension operator $\ex: W_2^{3/2}(\partial G)\to W_2^2(G)$, such that $\tr \circ \ex = \id$, cf.\ Triebel \cite{Tr83}. Hence, the function $f:= w-\poisol$ solves the inhomogeneous Dirichlet problem \eqref{D1} with $h=\tr w$ on $\partial G$.
On the other hand, let $f$ be the (unique) solution to \eqref{D1}. Since there exists a continuous linear extension operator from $W_2^{3/2}(\partial G)$ to $W_2^2(G)$ given by $\tilde{h}= \ex h$, we see that the function $\poisol := f-\tilde{h}$ satisfies \eqref{P1} with $g= \Laplace \tilde{h}$.

If the boundary $\partial G$ is Lipschitz the situation is different. It is known, see for example Jerison \& Kenig \cite[Theorem~B]{JK95}) that, in general, on a Lipschitz domain $G$ and for $g\in L_2(G)$ one can only expect that the solution $\poisol$ to \eqref{P1} belongs to $W_2^{3/2}(G)$; there are counterexamples of domains, for which $\poisol$ cannot be in $W_2^{\alpha}(G)$ for any $\alpha>3/2$. Thus, the above procedure does not work in a straightforward way.   However, by our strategy we can recover the negative result for this concrete domain, cf.\ Theorem~\ref{t2}:   \emph{If $g\in H_1(\real^2)\cap  W_2^1(\LL) $, then the solution $\poisol$ to \eqref{P1} is not in $W_2^{1+\sigma}(\LL)$ for any $\sigma\geq 2/3$}. Here $H_1(\real^2)\subset L_1(\real^2)$ is the Hardy space, cf.\ Stein \cite{stein}.

If $G$  is unbounded, the solution to \eqref{D1} might be not unique and, in general, it is only in the local space $W^2_{2,\mathrm{loc}}(G)$ even if $\partial G$ is smooth, cf.\ Gilbarg \& Trudinger \cite[Chapter 8]{gil-tru}. On the other hand,  if the complement $G^c$ is non-empty, if no component of $G^c$  reduces to a single point, and if the  boundary value $h$ is bounded and continuous on $\partial G$, then there exists a unique bounded solution to \eqref{D1} given by the convolution with the Poisson kernel, see Port \& Stone \cite[Theorem IV.2.13]{PS78}.

A strong motivation for this type of results comes from numerical analysis and approximation theory, because the exact Besov smoothness of $u$ is very important for computing $u$ and the feasibility of adaptive computational schemes, see  Dahlke \& DeVore \cite{DD97}, Dahlke, Dahmen \& DeVore \cite{DDD97}, DeVore \cite{De98}, Cohen, Dahmen \& DeVore \cite{CDD01}, Cohen \cite{C03}; an application to SPDEs is in Cioika et.\ al.\ \cite{Ci11,Ci14}. More precisely---using the set-up and the notation of \cite{CDD01}---let $\{\psi_\lambda, \,\lambda\in \Lambda\}$ be a basis of wavelets on $G$ and assume that the index set $\Lambda$ is of the form $\Lambda= \bigcup_{i\geq 0} \Lambda_i$ with (usually hierarchical) sets $\Lambda_i$ of cardinality $N_i$. By $u_{\Lambda_i}$ we denote the Galerkin approximation of $u$ in terms of the wavelets $\{\psi_\lambda\}_{\lambda\in \Lambda_i}$ (this amounts to solving a system of linear equations), and by $e_{N_i}(u):= \|u-u_{\Lambda_i}\|_p $ the approximation error in this scheme. Then it is known, cf. \cite[(4.2) and (2.35)]{CDD01}, that
\begin{equation}\label{Gal}
    u\in W_p^\sigma(G) \implies  e_{N_i}(u)  \leq C N_i^{-\sigma/d}, \quad i\geq 1.
\end{equation}
There is also an adaptive  algorithm for choosing the index sets $(\Lambda_i)_{i\geq 1}$.  Starting with an initial set $\Lambda_0$, this algorithm adaptively generates a sequence of nested sets $(\Lambda_i)_{i\geq 1}$; roughly speaking, in each iteration step we choose the next set $\Lambda_{i+1}$ by partitioning the domain of those wavelets $\psi_\lambda$, $\lambda\in \Lambda_i$ (i.e.\ selectively refining the approximation by considering the next generation of wavelets), whose coefficients $u_\lambda$ make, in an appropriate sense, the largest contribution to the sum $u=\sum_{\lambda\in \Lambda_i} u_\lambda \psi_\lambda$.

\runinhead{Notation.} Most of our notation is standard. By $(r,\theta)\in (0,\infty)\times (0,2\pi]$ we denote polar coordinates in $\real^2$, and $\mathbb{H}$ is the lower half-plane in $\real^2$. We write $f\asymp g$ to say that $c f(t)\leq g(t)\leq C f(t)$ for all $t$ and some fixed constants.

\section{Setting and the main result}\label{mainset}

Let $B = (B_t^x)_{t\geq 0}$ be a Brownian motion started at a point $x\in G$. Suppose that there exists a conformal mapping $\varphi: G\to \mathbb{H}$, where $\mathbb{H}:=\{(x_1,x_2)\in \mathbb{R}^2,\,\,x_2\leq 0\}$ is the lower half-plane in $\mathbb{R}^2$. Using the conformal invariance of Brownian motion, see e.g.\ M\"orters \& Peres \cite[p.~202]{MP10}, we can describe the distribution of the Brownian motion inside $G$ in terms of \emph{some} Brownian motion  $W$ in $\mathbb{H}$, which is much easier to handle. Conformal invariance of Brownian motion means that there exists a planar Brownian motion $W=(W_t^y)_{t\geq 0}$ with starting point  $y\in \mathbb{H}$ such that, under the conformal map $\varphi:G\to\mathbb{H}$  with boundary identification,
\begin{equation}\label{BW1}
    \left(\varphi(B_t^x)\right)_{0\leq t\leq \tau_G}
    \quad\text{has the same law as}\quad
    \left(W_{\xi(t)}^{\varphi(x)}\right)_{0\leq t \leq \tau_{\mathbb H}};
\end{equation}
the time-change $\xi$ is given by $\xi(t):= \int_0^t |\varphi'(B^x_s)|^2\,{\D}s$; in particular, $\xi(\tau_G)= \tau_\mathbb{H}$, where $\tau_G = \inf\{t>0: B_t^x\in\partial G\}$ and $\tau_\mathbb{H}:= \inf\{ t>0: \, W_t^{\varphi(x)}\in \partial \mathbb{H}\}$ are the first exit times from $G$ and $\mathbb H$, respectively.

Let us recall some properties of a planar Brownian motion in $\mathbb{H}$ killed upon exiting at the boundary $\partial\mathbb{H} = \{(w_1,w_2): \,w_2=0\}$. The distribution of the exit position $W_{\tau_\mathbb{H}}$  has the transition  probability density
\begin{equation}\label{WH}
    u\mapsto p_\mathbb{H}(w,u)
    = \frac{1}{\pi} \frac{|w_2|}{|u-w_1|^2 + w_2^2},\quad w=(w_1,w_2)\in \mathbb{H},
\end{equation}
cf.\ Bass~\cite[p.~91]{Bass}. Recall that a random variable $X$ with values in $\real$ has a Cauchy distribution, $X\sim\Cauchy(m,b)$, $m\in \real$, $b>0$,  if it has a transition probability density of the form
$$
    p(u)= \frac{1}{\pi}\frac{b}{(u-m)^2 +b^2}, \quad u\in\real;
$$
if $X\sim \Cauchy(m,b)$, then $Z:=(X-m)/b\sim \Cauchy(0,1)$. Thus, the probabilistic interpretation of $W_{\tau_\mathbb{H}}^w$ is
\begin{equation}\label{Cauchy}
    W_{\tau_\mathbb{H}}^w \sim Z^w\sim \Cauchy(w_1,|w_2|)
    \quad\text{or}\quad
    W_{\tau_\mathbb{H}}^w \sim \frac{Z-w_1}{|w_2|}
    \quad\text{where}\quad Z\sim \Cauchy(0,1).
\end{equation}

This observation allows us to simplify the calculation of functionals $\Theta$ of a Brownian motion $B$ on $G$, killed upon exiting from $G$, in the following sense:
\begin{equation}\label{fG0}\begin{aligned}
    \Ee \Theta(B_{\tau_G}^x)
    = \Ee \left(\Theta\circ \varphi^{-1}\right)(\varphi(B_{\tau_G}^x))
    &= \Ee\left(\Theta\circ \varphi^{-1}\right) (W_{\tau_\mathbb{H}}^{\varphi(x)})\\
    &= \Ee\left(\Theta\circ \varphi^{-1}\right) \left( \frac{Z-\varphi_1(x)}{|\varphi_2(x)|}\right).
\end{aligned}\end{equation}
In particular, the formula \eqref{fG0} provides us with a probabilistic representation for the solution $f$ to the Dirichlet problem \eqref{D1}:
\begin{equation}\label{fG}
    f(x)= \Ee h(B_{\tau_G}^x)= \Ee\left(h\circ \varphi^{-1}\right) (W_{\tau_\mathbb{H}}^{\varphi(x)}).
\end{equation}
\begin{remark}
    The formulae in \eqref{fG0} are very helpful for the numerical calculation of the values $\Ee \Theta(B_{\tau_G}^x)$. In fact, in order to simulate $\Theta(B_{\tau_G}^x)$, it is enough to simulate the Cauchy distribution $Z\sim \Cauchy(0,1)$ and then evaluate \eqref{fG0} using the  Monte Carlo method.
\end{remark}

We will now consider the $L$-shaped domain $\LL$. It is easy to see that the conformal mapping of $\LL$ to $\mathbb{H}$ is given by
\begin{equation}\label{Conf1}
    \varphi(z)= \eul^{\imag\frac{2\pi}{3}} z^{2/3} = r^{2/3} \exp\left( \tfrac{2}{3} \imag (\theta+\pi)\right)=\varphi_1(r,\theta)+ \imag \varphi_2(r,\theta),
\end{equation}
cf.\ Figure~\ref{fig-conf}, where $\theta = \arg z \in (0,2\pi]$.
\begin{figure}\label{fig-conf}
\centering
    \caption{Conformal mapping from $\LL$ to $\mathbb{H}$ and its behaviour at the boundaries.}
    \includegraphics[scale=0.8]{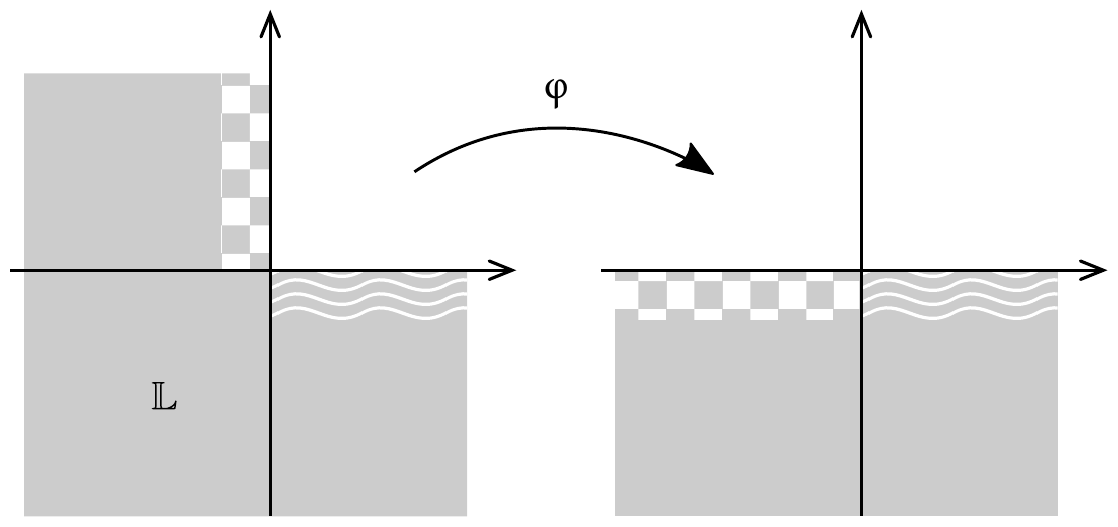}
\end{figure}

The following lemma uses the conformal mapping $\varphi: \LL\to \mathbb{H}$ and the conformal invariance of Brownian motion to obtain the distribution of $B^x_{\tau_\LL}$.
\begin{lemma}\label{tL}
    Let $\LL$ be an $L$-shaped domain as shown in Fig.~\ref{L-shape1}. The exit position $B_{\tau_\LL}$ of Brownian motion from $\LL$ is a random variable on $\partial \LL= \{0\}\times [0,\infty)\cup[0,\infty)\times \{0\}$ which has the following probability distribution:
    \begin{equation}\label{PBL}
    \begin{aligned}
        \Pp\left(B^x_{\tau_\LL}\in {\D}y\right)
        &= \frac{1}{\pi} \frac{|\varphi_2(x)|}{|\varphi_1(x)+y_2|^2+|\varphi_2(x)|^2} \,{\D}y_2 \,\delta_0({\D}y_1)\\
        &\qquad\mbox{} + \frac{1}{\pi}\frac{|\varphi_2(x)|}{|\varphi_1(x)-y_1|^2+|\varphi_2(x)|^2} \,{\D}y_1 \,\delta_0({\D}y_2).
    \end{aligned}
    \end{equation}
\end{lemma}

Lemma~\ref{tL} provides us with an explicit representation of the solution $f(x)$ to the Dirichlet problem \eqref{D1} for $G=\LL$. Indeed, since
(cf.\ Figure~\ref{fig-conf})
\begin{equation*}
    \left(h\circ \varphi^{-1}\right)(u)
    =
    \begin{cases}
        h(u,0), & u\geq 0, \\
        h(0,-u), &u\leq 0,
    \end{cases}
    \end{equation*}
we get
\begin{equation}\label{fD2}
    f(x)
    = \int_\real f_0(u) \,p_\mathbb{H}(\varphi(x),u) \,{\D}u
    = \frac{1}{\pi} \int_\real f_0(u)\, \frac{|\varphi_2(x)|}{(\varphi_1(x)-u)^2+ |\varphi_2(x)|^2}\,{\D}u,
\end{equation}
where
\begin{equation}\label{f0}
    f_0(u)
    := h(0,-u)\Id_{(-\infty,0)}(u) + h(u,0)\Id_{[0,\infty)}(u).
\end{equation}
After a change of variables, this becomes
\begin{equation}\label{fD22}
    f(x)
    =  \frac{1}{\pi} \int_\real f_0\left( u |\varphi_2(x)|+ \varphi_1(x)\right) \frac{{\D}u}{u^2+1}.
\end{equation}
If we want to investigate the smoothness of $f$, it is more convenient to rewrite $f$ in polar coordinates. From the right-hand side of \eqref{Conf1} we infer
\begin{equation}\label{Phi1}
    \varphi_1(r,\theta)
    = r^{2/3} \cos \Phi_\theta
    \quad\text{and}\quad
    \varphi_2(r,\theta)= r^{2/3} \sin \Phi_\theta,
\end{equation}
where we use the shorthand
\begin{equation*}
    \Phi_\theta:= \frac{2}{3}(\pi+\theta).
\end{equation*}
Observe that for $\theta\in (\pi/2, 2\pi]$ we have  $\pi< \Phi_\theta\leq 2\pi$, hence $\varphi_2  \leq 0$. This yields
\begin{equation}\label{f1-0}
    f(r,\theta)
    = \frac{1}{\pi}\int_\real f_0\left(r^{2/3} \cos \Phi_\theta  - r^{2/3}v\sin \Phi_\theta\right)\, \frac{{\D}v}{1+v^2}.
\end{equation}

Now we turn to the principal objective of this note: the smoothness of $f$ in the Sobolev--Slobodetskij scale.
\begin{theorem}\label{t1}
    Consider the \textup{(}homogeneous\textup{)} Dirichlet problem \eqref{D1} with a boundary term $f_0$, given by \eqref{f0}, and let $f$ denote the solution to \eqref{D1}.
     \begin{enumerate}\renewcommand{\theenumi}{\textup{\alph{enumi}}}
     \item\label{t1-a}
        If $f_0\in W_1^2(\real) \cap  W_2^2(\real)$ satisfies
        \begin{equation}\label{f0}
            \liminf_{\epsilon \to 0} \int_{|x|>\epsilon} \frac{f_0'(x)}{x} \,{\D}x\neq 0,
        \end{equation}
        then $f\notin W_2^{1+\sigma}(\LL)$, even $f\notin W_{2,\mathrm{loc}}^{1+\sigma}(\LL)$, for any $\sigma\geq 2/3$.

     \item\label{t1-b}
        If $f_0\in W_1^2(\real)\cap  W_p^1(\real)$, where $p>\max\{2,\, 2/(2-3\sigma)\}$, then $f\in W_{2,\mathrm{loc}}^{1+\sigma}(\LL)$ for all  $\sigma\in (0,2/3)$.
     \end{enumerate}
\end{theorem}
\begin{remark}\label{rem1}
    By the Sobolev embedding theorem we have $W_1^2(\real) \cap W_2^2(\real) \subset C_b(\real)$ and $W_1^2(\real)\cap  W_p^1(\real)\subset C_b(\real)$ if $p>\max\{2,\, 2/(2-3\sigma)\}$. Hence, the function $f$ given by \eqref{fD22} is the unique bounded solution to \eqref{D1}.
\end{remark}
The idea of the proof of Theorem~\ref{t1} makes essential use of the results by Jerison \& Kenig \cite{JK95} combined with the observation that it is, in fact, enough to show the claim for $\LLB:=\LL\cap B(0,1)$, where $B(0,1):= \{ x\in \mathbb{R}^2:\, |x|< 1\}$.

Theorem~\ref{t1} allows us to prove the negative result for the solution to the Poisson problem, which improves \cite[Theorem~B]{JK95}.
Recall that  $H_1(\real^2)\subset L_1(\real^2)$  is   the usual Hardy space, cf.\ Stein \cite{stein}.
\begin{theorem}\label{t2}
    Consider the Poisson \textup{(}inhomogeneous Dirichlet\textup{)} problem \eqref{P1} with right-hand side  $g\in H_1(\real^2)\cap W_2^1(\LL)$ such that $f_0(x):=\left((\tr g*N)\circ \varphi^{-1}\right) (x)$ satisfies \eqref{f0}, where $N(x)= (2\pi)^{-1}\log |x|$ is the Newton kernel. Then the solution $\poisol \notin W_2^{1+\sigma} (\LL)$, even $\poisol \notin W_{2,\mathrm{loc}}^{1+\sigma} (\LL)$, for any $\sigma\geq 2/3$.
\end{theorem}
The proofs of Theorem~\ref{t1} and \ref{t2} are deferred to the next section.

\section{Proofs}\label{proofs}

\begin{proof}[Proof of Lemma~\ref{tL}]  We calculate the characteristic function of $B^x_{\tau_\LL}$.
As before, let $y=(y_1,y_2)$, $x=(x_1,x_2)$ and $\varphi(x)=(\varphi_1(x),\varphi_2(x))$. We have
\begin{align*}
    \Ee \eul^{\imag\xi\cdot B_{\tau_\LL}^x}
    &\overset{\eqref{BW1}}{=} \Ee \eul^{\imag\xi \cdot\varphi^{-1} (W_{\tau_\mathbb{H}}^{\varphi(x)})} \\
    &= \int_{\mathbb{R}^2}  \eul^{\imag\xi\cdot \varphi^{-1} (y)} \,\Pp(W_{\tau_\mathbb{H}}^{\varphi(x)}\in {\D}y)\\
    &\overset{\eqref{WH}}{=} \frac{1}{\pi} \int_\real \eul^{\imag\xi \cdot \varphi^{-1}(y_1,0)}\frac{|\varphi_2(x)|}{|\varphi_1(x)-y_1|^2+|\varphi_2(x)|^2}{\D}y_1\\
    &= \frac{1}{\pi} \int_{-\infty}^0 \eul^{-\imag \xi_2u} \frac{|\varphi_2(x)|}{|\varphi_1(x)-u|^2+|\varphi_2(x)|^2}\,{\D}u\\
    &\qquad\mbox{} + \frac{1}{\pi} \int_0^{+\infty} \eul^{\imag \xi_1u }\frac{|\varphi_2(x)|}{|\varphi_1(x)- u|^2+|\varphi_2(x)|^2}\,{\D}u.\tag*{\qed}
\end{align*}
\end{proof}

For the proof of Theorem~\ref{t1} we need some preparations. In order to keep the presentation self-contained, we quote the classical result by Jerison \& Kenig \cite[Theorem~4.1]{JK95}.
\begin{theorem}[Jerison \& Kenig]\label{JK}
    Let $\sigma\in (0,1)$, $k\in \nat_0$ and $p\in [1, \infty]$. For any function $u$ which is harmonic on a bounded domain $\Omega$, the following assertions are equivalent:
    \begin{enumerate}\renewcommand{\theenumi}{\textup{\alph{enumi}}}
    \item\label{JK-a}
        $f\in B_{pp}^{k+\sigma}(\Omega)$;
    \item\label{JK-b}
        $\dist(x,\partial\Omega)^{1-\sigma}\,\left| \nabla^{k+1} f\right| + \left| \nabla^k f \right| + \left| f\right| \in L_p(\Omega)$.
    \end{enumerate}
\end{theorem}

We will also need the following technical lemma. Recall that $\LLB = \LL\cap B(0,1)$.
\begin{lemma}\label{B1}
    Suppose that $f_0\in W_p^1(\real)$ for some $p>2$. Then $f\in  W_2^1 (\LLB)$.
\end{lemma}
\begin{proof}
    Using  the representation \eqref{f1-0}, the H\"older inequality and a change of variables,  we get
\begin{align*}
    &\int_{\pi/2}^{2\pi} \int_0^1 |f(r,\theta)|^2 r\,{\D}r\,{\D}\theta\\
    &=  \frac{3}{2\pi^2 }\int_{\pi/2}^{2\pi} \int_0^1  \rho^2
        \left|\int_\real f_0(w)\frac{|\rho \sin  \Phi_\theta|}{(w-\rho \cos  \Phi_\theta)^2+ (\rho \sin  \Phi_\theta)^2}{\D}w\right|^2 \,{\D}\rho\,{\D}\theta\\
    &\leq C_1\int_{\pi/2}^{2\pi} \int_0^1  \rho^2
        \left[
        \left(\int_\real \left| f_0(v)\right|^p{\D}v\right)^{1/p}
        \left(\int_\real \frac{|\rho \sin  \Phi_\theta|^q}{(v^2 + |\rho \sin  \Phi_\theta|^2)^q}{\D}v \right)^{1/q}
        \right]^2\,{\D}\rho\,{\D}\theta\\
    &\leq C_2\int_{\pi/2}^{2\pi} \int_0^1  \rho^2 \left|\rho \sin  \Phi_\theta\right|^{-2+2/q}
        \left(\int_\real \frac{1}{(w^2 +1)^q}{\D}w \right)^{2/q}\,{\D}\rho\,{\D}\theta\\
    &= C_3 \int_{\pi/2}^{2\pi} \int_0^1 \rho^{2/q} \left|\sin  \Phi_\theta\right|^{-2+2/q} \,{\D}\rho\,{\D}\theta,
\end{align*}
where $p^{-1}+q^{-1}=1$. Because of $p>2$ we have $-2+2/q>-1$, hence $q<2$. Note that the inequalities  $2x/\pi \leq \sin x\leq x$ for $x\in [0,\pi/2]$, imply
$$
    \int_{\pi/2}^{2\pi} \left|\sin \Phi_\theta\right|^{-1+\epsilon} \,{\D}\theta
    = \int_0^{\pi} \left|\sin \varphi\right|^{-1+\epsilon} \,{\D}\varphi
    = 2 \int_0^{\pi/2} \left|\sin\varphi\right|^{-1+\epsilon}\,{\D}\varphi
    < \infty.
$$
This shows that $f\in L_2(\LLB)$.

Recall that the partial derivatives of the polar coordinates are
\begin{equation}\label{der10}
    \frac{\partial}{\partial x_1}  r = \cos \theta, \quad
    \frac{\partial}{\partial x_1}  \theta= - \frac{\sin \theta}{r}, \quad
    \frac{\partial}{\partial x_1}  \Phi_\theta= \frac{2}{3} \frac{\partial}{\partial x_1}  \theta =  - \frac{2\sin \theta}{3r}.
\end{equation}
Therefore, we have for $\theta\in (\pi/2,2\pi)$
\begin{align}
\notag
    \frac{\partial}{\partial x_1} &f(r,\theta)
    = \frac{1}{\pi} \int_\real f_0' \left(  r^{2/3} \cos \Phi_\theta- vr^{2/3} \sin \Phi_\theta \right) \, \frac{1}{v^2+1}\times\\
    &\notag\quad \mbox{}\times \left[\frac{2\cos \theta }{3 r^{1/3} } \left(  \cos \Phi_\theta- v \sin \Phi_\theta \right)\right.
    \\
    &\label{f1}\quad\qquad\left.\mbox{}+r^{2/3} \left( \frac{-2\sin \theta}{3r}\right) ( -v \cos \Phi_\theta -\sin \Phi_\theta )  \right] {\D}v\\
    &\notag= \frac{2}{3 \pi r^{1/3} } \int_\real f_0'\left(  r^{2/3} \cos \Phi_\theta- vr^{2/3} \sin \Phi_\theta \right)  \, \frac{1}{v^2+1}\times\\
    &\notag\quad \mbox{}\times \left[\left(   \cos \Phi_\theta- v \sin \Phi_\theta \right) \cos \theta
     +( v \cos \Phi_\theta +\sin \Phi_\theta ) \sin \theta  \right] {\D}v\\
    &\notag= \frac{2}{3\pi r^{1/3} } \int_\real f_0' \left(  r^{2/3} \cos \Phi_\theta- vr^{2/3} \sin \Phi_\theta \right) \, \frac{K(\theta,v)}{v^2+1}\,{\D}v,
\end{align}
where
\begin{equation}\label{Kav}
K(\theta,v):=\cos \omega_\theta - v \sin \omega_\theta ,
\end{equation}
and
\begin{equation}\label{tha2}
\omega_\theta= \frac{1}{3} \left(2\pi-\theta\right).
\end{equation}
Note that  $\Phi_{\pi/2} =\pi$ and $\omega_{\pi/2}= \pi/2$.

Let us show that the first partial derivatives of $f$ belong to $L_2(\LLB)$. Because of the symmetry of $\LLB$, is it enough to check this for $\frac{\partial}{\partial x_1} f$.

Using the estimate $|K(\theta, v)| (1+v^2)^{-1} \leq C(1+|v|)^{-1}$, a change of variables and the H\"older inequality, we get
\begin{align*}
    \int_0^1 &\int_{\pi/2}^{2\pi} \left| \frac{\partial}{\partial x_1}  f(r,\theta)\right|^2 r\,{\D}\theta\,{\D}r\\
    &=  \int_0^ 1 \int_{\pi/2}^{2\pi}
        \left|
        \int_\real \frac{2}{3 \pi r^{1/3} }\, f_0'
        \left(r^{2/3} \cos \Phi_\theta- vr^{2/3} \sin \Phi_\theta \right)
        \frac{K(\theta, v)}{1+v^2} \,{\D}v
        \right|^2
        r\,{\D}\theta \,{\D}r\\
    &=  \frac{2}{3\pi^2} \int_0^1 \int_{\pi/2}^{2\pi} \rho
        \left|
        \int_\real f_0'\left(\rho \cos \Phi_\theta- v\rho\sin \Phi_\theta \right) \frac{K(\theta, v)}{1+v^2}\,{\D}v
        \right|^2
        \,{\D}\theta\,{\D}\rho\\
    &\leq C_1 \int_0^1 \int_{\pi/2}^{2\pi}\rho
        \left(\int_\real \frac{|f_0'(w)|}{|\rho \sin \Phi_\theta| +|w- \rho \cos \Phi_\theta|}{\D}w \right)^2
        \,{\D}\theta \,{\D}\rho\\
    &\leq C_2 \left(\int_\real \left| f_0' (w)\right|^p {\D}w\right)^{2/p}
    \left( \int_\real  \frac{1}{(1+|w|)^q} {\D}w\right)^{2/q} \times\\
    &\qquad\mbox{}\times \int_{\pi/2}^{2\pi} \int_0^1\rho  \left( |\rho \sin \Phi_\theta|^{-1+1/q}\right)^2\,{\D}\rho\,{\D}\theta\\
    &=  C_3 \int_{\pi/2}^{2\pi}\int_0^1  |\sin \Phi_\theta|^{-2+2/q}  \rho^{-1+2/q}  \,{\D}\rho \,{\D}\theta
    <\infty;
\end{align*}
in the last line we use again that $-2+2/q>-1$.
\smartqed\qed
\end{proof}

\begin{proof}[Proof of Theorem~\ref{t1}]  It is enough to consider the set $\LLB$.
We verify that condition~\ref{JK-b} of Theorem~\ref{JK} holds true. We check whether
$$
    \dist(0,\cdot)^{1-\sigma} \left| \frac{\partial^2}{\partial x_1^2}   f \right| + \left| \frac{\partial}{\partial x_1}  f\right|+\left|f\right|
    \quad\text{is in $L_2(\LLB)$ or not.}
$$
From Lemma~\ref{B1} we already know that $\left| \frac{\partial}{\partial x_1}  f\right|+\left|f\right|\in  L_2(\LLB)$. Let us check when
$$
    \dist(0,\cdot)^{1-\sigma} \left| \frac{\partial^2}{\partial x_1^2} f \right| \in L_2(\LLB).
$$
We will only work out the term $\frac{\partial^2}{\partial x_1^2} f(r,\theta)$ since the calculations for $\tfrac{\partial^2}{\partial x_1\partial x_2} f(r,\theta)$ are similar. We have
\begin{align*}
    \frac{\partial}{\partial x_1}  K(\theta,v)
    = \frac{\sin \theta}{3r}  \left(- v \cos \omega_\theta - \sin \omega_\theta\right)
    =: \frac{\sin \theta}{3r} K^{\star}(\theta,v),
\end{align*}
where use that $\tfrac{\partial}{\partial x_1}  \omega_\theta= -\frac{1}{3} \tfrac{\partial}{\partial x_1} \theta= \frac{ \sin \theta}{3r}$ and set
\begin{equation}\label{K1}
    K^{\star}(\theta,v):= -v \cos \omega_\theta - \sin \omega_\theta.
\end{equation}
Therefore, differentiating $\tfrac{\partial}{\partial x_1}  f$---we use the representation \eqref{f1}---with respect to $x_1$  gives
\begin{align*}
    \tfrac{\partial^2}{\partial x_1^2} f(r,\theta)
    &= - \frac{2\cos \theta}{9  \pi r^{4/3}}  \int_\real  f_0' \left( r^{2/3} (\cos \Phi_\theta- v \sin \Phi_\theta)\right)\frac{K(\theta, v)}{1+v^2}  \,{\D}v\\
    &\qquad \mbox{}+\frac{4}{9 \pi r^{2/3}}    \int_\real  f_0'' \left( r^{2/3} (\cos \Phi_\theta- v \sin \Phi_\theta)\right)\frac{K^2(\theta, v)}{1+v^2}  \,{\D}v\\
    &\qquad \mbox{}+\frac{2\sin \theta}{9  \pi r^{4/3}}   \int_\real  f_0' \left( r^{2/3} (\cos \Phi_\theta- v \sin \Phi_\theta)\right)\frac{K^{\star}(\theta, v)}{1+v^2}  \,{\D}v.
\end{align*}

Note that
\begin{equation}\label{int2}\begin{aligned}
    &\int_{\LLB} \dist (x, \partial \LLB)^{2-2\sigma} \left|\frac{\partial^2}{\partial x_1^2} f(x)\right|^2  \,{\D}x\\
    &\qquad= \int_0^1 \int_{\pi/2}^{2\pi} \dist ((r,\theta), \partial \LLB)^{2-2\sigma} \left|\frac{\partial^2}{\partial x_1^2}f(r,\theta)\right|^2 r\,{\D}\theta\,{\D}r.
\end{aligned}\end{equation}
Since only the values near the boundary $\Gamma:= \partial \LLB\cap \partial \LL$ determine the convergence of the integral, it is enough to  check that
\begin{equation}\label{int22}
    \mathbb{I}
    = \int_0^1 \int_{\pi/2}^{2\pi} \dist ((r,\theta), \Gamma)^{2-2\sigma} \left|\frac{\partial^2}{\partial x_1^2} f(r,\theta)\right|^2 r\,{\D}\theta\,{\D}r
\end{equation}
is infinite if $\sigma\geq 2/3$ and finite if $\sigma<2/3$.

\enlargethispage{.66\baselineskip}
We split $\LLB$ into three parts. For $\delta>0$ small enough we define, see Figure~\ref{k1-k3},
\begin{figure}\label{k1-k3}
\sidecaption
    \includegraphics[scale=0.8]{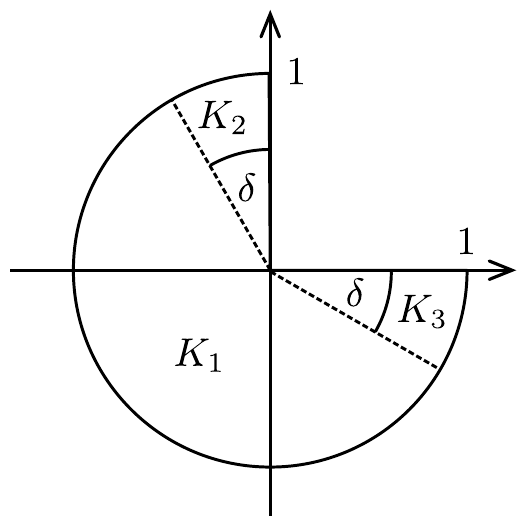}
    \caption{The set $\LLB$ is split into three disjoint parts $K_1$, $K_2$, $K_3$.}
\end{figure}
\begin{gather*}
    K_1:= \left\{ (r,\theta):\, 0<r<1, \;\; \frac{\pi}{2}+\delta<\theta<2\pi-\delta\right\},\\
    K_2:= \left\{ (r,\theta):\, 0<r<1, \;\; \frac{\pi}{2} \leq \theta<\frac{\pi}{2}+\delta\right\},\\
    K_3:= \left\{ (r,\theta):\, 0<r<1, \;\; 2\pi-\delta <\theta\leq 2\pi\right\}.
\intertext{Splitting the integral accordingly, we get}
    \mathbb{I}
    = \left(\int_{K_1}+\int_{K_2}+\int_{K_3}\right)\dist ((r,\theta), \Gamma)^{2-2\sigma} \left|\frac{\partial^2}{\partial x_1^2} f(r,\theta)\right|^2 r\,{\D}\theta\,{\D}r;
\end{gather*}
in order  to show that $\mathbb{I}$ is infinite if $\sigma\geq 2/3$, it is enough to see that the integral over $K_1$  is infinite.
Noting that in $K_1$ we have $\dist((r,\theta), \Gamma)\asymp r$, we get
\begin{equation}\label{eq1}
\begin{aligned}
    \int_{K_1} &\left|r^{1-\sigma}\tfrac{\partial^2}{\partial x_1^2} f(r,\theta)\right|^2 r\,{\D}\theta \,{\D}r\\
    &= \int_{K_1} r\left|r^{1-\sigma}  \frac{2}{9 \pi   r^{4/3}} \right|^2 \times\\
    &\quad \mbox{}\times \left| \int_\real   f_0'  \left( r^{2/3} (\cos \Phi_\theta- v \sin \Phi_\theta)\right)
    \frac{K^{\star}(\theta, v)\sin \theta- K(\theta,v)\cos \theta}{1+v^2} \,{\D}v\right.\\
    &\qquad \left.\mbox{}+2r^{2/3}  \int_\real   f_0''  \left( r^{2/3} (\cos \Phi_\theta- v \sin \Phi_\theta)\right)\frac{K^2(\theta,v)}{1+v^2}\,{\D}v\right|^2 \,{\D}r \,{\D}\theta\\
    &= \frac{4}{81  \pi^2  } \int_{K_1}  r^{1/3-2\sigma}
        \left|  \int_\real   f_0'  \left( r^{2/3} (\cos \Phi_\theta- v \sin \Phi_\theta)\right)\frac{K^{\star\star}(\theta, v)}{1+v^2} \,{\D}v\right.\\
    &\qquad \left.\mbox{}+ 2r^{2/3}  \int_\real   f_0''\left( r^{2/3} (\cos \Phi_\theta- v \sin \Phi_\theta)\right)\frac{K^2(\theta,v)}{1+v^2} \,{\D}v\right|^2 \,{\D}r\,{\D}\theta\\
    &= \frac{4}{81  \pi^2  }\int_{K_1}  r^{1/3-2\sigma} \left|J(r^{2/3},\theta)+ I(r^{2/3},\theta)\right|^2 \,{\D}r\,{\D}\theta\\
    &= \frac{2}{27  \pi^2 }\int_{K_1} \rho^{1-3\sigma} \left|J(\rho,\theta)+ I(\rho,\theta)\right|^2 \,{\D}\rho\,{\D}\theta,
\end{aligned}
\end{equation}
where we use the following shorthand notation
\begin{gather*}
    K^{\star\star}(\theta,v)
    := K^{\star}(\theta, v)\sin \theta- K(\theta,v)\cos \theta
    = -v \sin (\theta-\omega_\theta)- \cos  (\theta-\omega_\theta),
\\
    J(\rho,\theta)
    := \int_\real f_0' \left(\rho(\cos \Phi_\theta- v \sin \Phi_\theta)\right)\frac{K^{\star\star}(\theta, v)}{1+v^2}\,{\D}v,
\\
    I(\rho,\theta)
    := 2\rho \int_\real f_0'' \left( \rho(\cos \Phi_\theta- v \sin \Phi_\theta)\right)\frac{K^2(\theta, v)}{1+v^2} \,{\D}v.
\end{gather*}
Observe that $\theta-\omega_\theta\in (0,2\pi)$ for $\theta\in (\frac{\pi}{2}, 2\pi)$, and
$\theta-\omega_\theta\in (\frac{4\delta}{3},2\pi-\frac{4\delta}{3})$ whenever $\theta\in  (\frac{\pi}{2}+\delta, 2\pi-\delta)$.

Without loss of generality we may assume that $J(\rho,\theta)+ I(\rho,\theta)\not\equiv 0$ on $K_1$. Let us show that
$\lim_{\rho\to 0} |J(\rho,\theta)+ I(\rho,\theta)| = C(f_0, \theta)>0$. This guarantees that we can choose some $K_{11}\subset K_1$ such that
\begin{equation}\label{K11}
    |J(\rho,\theta)+ I(\rho,\theta)|\geq C(f_0)>0 \quad \text{on $K_{11}$.}
\end{equation}

Using the change of variables $x=v\rho$ we get, using dominated convergence,
\begin{align*}
I(\rho,\theta)
&= 2\int_\real f_0''\left(\rho\cos \Phi_\theta- x \sin \Phi_\theta\right)\frac{\big(\rho \cos \omega_\theta -x \sin \omega_\theta)^2}{\rho^2+x^2} \,{\D}x\\
& \underset{\rho\to 0}{\longrightarrow} \,2\sin^2 \omega_\theta \int_\real f_0''\left(-x \sin \Phi_\theta\right) \,{\D}x=\frac{2\sin^2 \omega_\theta}{\sin \Phi_\theta} \int_\real f_0''(x) \,{\D}x =0,
\end{align*}
since we assume that $f_0\in W_2^1 (\real)$.

For $J(\rho,\theta)$ we have, using the same change of variables,
\begin{align*}
J(\rho, \theta)&= -\int_\real f_0'\left(\rho \cos \Phi_\theta- \rho v \sin \Phi_\theta\right)\frac{\cos (\theta-\omega_\theta) + v \sin(\theta-\omega_\theta)}{1+v^2} \,{\D}v\\
&= -\int_\real f_0'\left(\rho \cos \Phi_\theta- x \sin \Phi_\theta\right)\frac{\rho\cos (\theta-\omega_\theta) + x \sin(\theta-\omega_\theta)}{\rho^2+x^2} \,{\D}x\\
&= - \left( \int_{|x|>\epsilon} +\int_{|x|\leq \epsilon} \right) \left( \dots\right) \,{\D}x.
\end{align*}
The first integral can be treated with the dominated convergence theorem because we have $f_0'\in L_1(\real)$ and
$\rho (\rho^2+x^2)^{-1} \leq x^{-2}$, $x(\rho^2+x^2)^{-1} \leq x^{-1}$ are bounded for $|x|>\epsilon$. Therefore,
$$
    \lim_{\rho\to 0} \left[-  \int_{|x|>\epsilon}\left(\dots \right)\,{\D}x\right]
    = - \sin(\theta-\omega_\theta)\int_{|x|>\epsilon}\frac{f'_0(-x\sin \Phi_\theta)}{x}\,{\D}x.
$$
Now we estimate the two parts of the second integral. For
$$
    -\int_{|x|\leq \epsilon} f_0'(\rho \cos \Phi_\theta-x\sin \Phi_\theta)
    \frac{\rho\cos(\theta-\omega_\theta)}{\rho^2+x^2}\,{\D}x
$$
we have
$\rho(\rho^2+x^2)^{-1}\leq \epsilon^{-1} \rho x(\rho^2+x^2)^{-1} \leq \epsilon^{-1}$,
so this term tends to $0$ by the dominated convergence theorem.  For the second term in this integral we have using a change of variables and the Cauchy--Schwarz inequality,
\begin{align*}
  |\sin& (\theta-\omega_\theta)|\cdot   \left| \int_{|x|\leq \epsilon} f_0'
 \left( \rho \cos\Phi_\theta -x \sin \Phi_\theta\right) \frac{x}{x^2+\rho^2}
 dx\right| \\
 & \leq
 \left|\int_{|w|\leq \epsilon} \left(f_0' \left( \rho \cos \Phi_\theta- w \sin \Phi_\theta \right)- f_0' \left( \rho \cos \Phi_\theta\right)\right)\frac{w}{\rho^2+w^2}\,{\D}w\right|
 \\
    &\leq  \int_{|w|\leq \epsilon} \int_0^1 |f_0''(\rho \cos \Phi_\theta- r w \sin \Phi_\theta) | \,{\D}r \,{\D}w\\
    &\leq  \sqrt{2\epsilon} \int_0^1 \left(\int_\real |f_0''(\rho \cos \Phi_\theta- v \sin \Phi_\theta)|^2 \,{\D}v\right)^{1/2} \,{\D}r\\
    &\leq C_1(\theta) \sqrt\epsilon \,\|f_0''\|_2.
\end{align*}
Altogether we have  upon letting $\rho \to 0$ and then $\epsilon\to 0$, that
\begin{equation}\label{I10}
    \lim_{\rho\to 0} I (\rho,\theta) = 0,
\end{equation}
\begin{equation}\label{I10}
   \liminf_{\epsilon \to 0} \lim_{\rho\to 0} J (\rho,\theta)
    =\sin (\omega_\theta-\theta)  \liminf_{\epsilon\to 0} \int_{|x|>\epsilon} \frac{f_0'(x)}{x}\,{\D}x.
\end{equation}
If the ``$\liminf$'' diverges, it is clear that \eqref{K11} holds, if it converges but is still not equal to 0, we can choose $K_{11}$ in such a way that $\sin (\omega_\theta-\theta)\neq 0$.  Thus, the integral over $K_1$ blows up as $\int_0^1 \rho^{1-3\sigma}\,{\D}\rho=\infty$ for any $\sigma\geq 2/3$.

\medskip
To show the convergence result, we have to estimate $I$ and $J$ from above.
Write
\begin{align*}
 J(\rho,\theta)
 =& -\int_\real f_0'\left(\rho (\cos \Phi_\theta- \nu \sin \Phi_\theta)\right)\frac{\nu\sin (\theta-\omega_\theta)}{1+\nu^2} \,{\D}\nu  \\
  &\mbox{} -\int_\real   f_0'\left(\rho ( \cos \Phi_\theta- \nu \sin \Phi_\theta)\right)
  \frac{\cos(\theta-\omega_\theta)}{1+\nu^2} \,{\D}\nu =: J_1(\rho,\theta) + J_2(\rho,\theta).
\end{align*}
Since $f_0\in W_p^1(\real)$, using the H\"older inequality and a change of variables give
\begin{equation}\label{J2}
\begin{aligned}
    |J_{1}(\rho,\theta)|
    &\leq \left(\int_\real \left| f_0' \left( \rho(\cos \Phi_\theta- v \sin \Phi_\theta)\right)\right|^p\,{\D}v\right)^{\frac 1p}\left( \int_\real \left(\frac{v}{1+v^2}\right)^q \,{\D}v\right)^{\frac 1q}\\
    &\leq c  |\rho \sin \Phi_\theta|^{-1/p}
\end{aligned}
\end{equation}
for all $\theta \in [\pi/2,2\pi]$ and $\rho>0$. An even simpler calculation yields
\begin{equation}\label{J1}
    |J_{2}(\rho,\theta)|\leq c  |\rho \sin \Phi_\theta|^{-1/p}
\end{equation}
for all $\theta \in [\pi/2,2\pi]$ and $\rho>0$. Now we estimate $I(\rho,\theta)$.
Note that for every $\theta\in [\pi/2, 2\pi]$ we have $K^2 (\theta,\nu)/(1+\nu^2) \leq C$. By a change of variables we get
\begin{equation}\label{i1}
   \left| I(\rho, \theta)\right|
   \leq \frac{C_1}{|\sin \Phi_{\theta}|}  \int_\real  \left|f_0''(w+\rho  \cos \Phi_\theta )\right|\,{\D}w
   \leq\frac{C_2}{|\sin \Phi_{\theta}|}
\end{equation}
for all $\theta \in [\pi/2,2\pi]$ and $\rho>0$. Note that for  $\Phi_\theta\in [\pi + 2\delta/3, 2\pi-2\delta/3]$ it holds that $|\sin\Phi_\theta|>0$.  Thus, on $K_1$ we have
\begin{equation}\label{JK-K1}
|I(\rho, \theta)+ J(\rho, \theta)| \leq C\rho^{-1/p},
\quad \theta \in [\pi/2+\delta,2\pi-\delta],\;\rho>0,
\end{equation}
implying
$$
    \int_{ K_1}\left|r^{1-\sigma}\frac{\partial^2}{\partial x_1^2} f(r,\theta)\right|^2 r \,{\D}\theta \,{\D}r
    \leq C \int_0^1 \rho^{1-3\sigma-2/p} \,{\D}\rho.
$$
The last integral converges if $\sigma\in (0,2/3)$ and $p>\frac{2}{2-3\sigma}$.

In order to complete the proof of the convergence part, let us show that the integrals over  $K_2$ and $K_3$ are convergent for all $\sigma\in (0,1)$.

In the regions  $K_2$ and $K_3$ we have $\dist ((r,\theta), \Gamma)\leq r |\cos \theta|$ and $\dist ((r,\theta), \Gamma)\leq r |\sin \theta|$, respectively.  We will discuss only $K_2$ since $K_3$ can be treated in a similar way. We need to show that
\begin{equation}\label{K2-1}
    \int_{K_2}  \left|  |r\cos \theta| ^{1-\sigma}  \frac{\partial^2}{\partial x_1^2}  f(r,\theta) \right|^2 r\,{\D}r\,{\D}\theta
    < \infty
    \quad\text{for all\ } \sigma\in (0,1).
\end{equation}

From \eqref{J2},  \eqref{J1}  and \eqref{i1} we derive that for all $(\rho,\theta)\in \LLB$
\begin{equation}\label{j3}
    \left|J(\rho, \theta)+I(\rho,\theta) \right|
    \leq C \rho^{-\frac 1p} \left( |\sin \Phi_\theta|^{-1}+ |\sin \Phi_\theta|^{-\frac 1p}\right)
    \stackrel{}{\leq}{} C' \rho^{-\frac 1p}|\sin \Phi_\theta|^{-1}.
\end{equation}
Now we can use a calculation similar to \eqref{eq1} for $K_1$ to show that \eqref{K2-1} is finite and, therefore, it is enough to show that
\begin{equation}\label{K2-2}
    \int_{\frac{\pi}{2}}^{\frac{\pi}{2}+\delta}  \left( \frac{|\cos \theta|^{1-\sigma}}{\sin\Phi_\theta} \right)^2 {\D}\theta
    < \infty.
\end{equation}
Observe that $\lim_{\theta\to \frac{\pi}{2}} \cos \frac13(\pi+ \theta) / \cos \theta=\frac{1}{3}$, implying
$$
    \frac{|\cos \theta|^{1-\sigma}}{\sin\Phi_\theta}
    = \frac{| \cos \theta|^{1-\sigma}}{2 \sin \frac 13(\pi+ \theta) \cos \frac 13(\pi+ \theta)}
     \asymp |\cos\theta|^{-\sigma}
     \quad\text{as $\theta\to \dfrac{\pi}{2}$}.
$$
Therefore, it is sufficient to note that for any $\sigma \in (0,1)$
$$
    \int_{\pi/2}^{\pi/2+\delta} |\cos \theta|^{-2\sigma} \,{\D}\theta
    \asymp \int_0^1 \frac{{\D}x}{(1-x^2)^\sigma}
    =  \int_0^1 \frac{{\D}x}{(1-x)^\sigma(1+x)^\sigma}
    < \infty.
$$

Summing up, we have shown that
\begin{gather*}
    \dist(0,\cdot)^{1-\sigma} \left| \frac{\partial^2}{\partial x_1^2}   f \right| \in L_2(\LLB)
    \quad\text{resp.}\quad
    \notin L_2(\LLB),
\end{gather*}
according to $\sigma\in (0,2/3)$ or $\sigma\in [2/3,1)$.
\end{proof}

\begin{proof}[Proof of Theorem~\ref{t2}]
Let $\poisol$ be the solution to \eqref{P1} on $\LL$ with source function $g$, and define $w= g* N$ for the Newtonian potential $N$ on $\mathbb{R}^2$. As we have already mentioned in the introduction, $f:= w-\poisol$ is the solution to \eqref{D1} on $\LL$ with the boundary condition $h:= \tr w$ on $\partial \LL$.
Note that under the condition  $g\in H_1(\real^2)\cap  W_2^1(\LL)$   we have  $\Delta w= g$ (cf.\ Stein~\cite[Theorem III.3.3, p.~114]{stein}), which implies $w\in  W_1^{3}(\real^2)\cap W_2^{3}(\LL)$.  By the trace theorem we have $h\in W_1^{2} (\partial \LL )\cap W_2^{5/2}(\partial \LL)$, which in terms of $f_0$ means  $f_0\in W_1^{2} (\real)\cap W_2^{5/2}(\real)$.    The explosion result of Theorem~\ref{t1} requires  $f_0\in W_1^{2} (\real )\cap  W_2^2(\real)$ and \eqref{f0}.  The latter is guaranteed by the assumption on the trace in the statement of the theorem.   Hence, $f\notin W_{2,\mathrm{loc}}^{1+\sigma}(\LL)$, $\sigma\geq 2/3$. Since
$w\in  W_{2,\mathrm{loc}}^{2}(\LL)$, this implies that $F\notin  W_{2,\mathrm{loc}}^{1+\sigma}(\LL)$, $\sigma\geq 2/3$.
\smartqed\qed
\end{proof}

\subsection*{Acknowledgement}   We thank  S.\ Dahlke (Marburg) who pointed out the reference \cite{JK95},  N.\ Jacob (Swansea) for his suggestions on the representation of Sobolev--Slobodetskij spaces, and A.\ Bendikov (Wroc{\l}aw) who told us about the papers \cite{MM03}, \cite{MMY10}. We are grateful to B.\ B\"ottcher for drawing the illustrations and commenting on the first draft of this paper. Financial support from NCN grant
2014/14/M/ST1/00600 (Wroc{\l}aw) for V.~Knopova is gratefully acknowledged.

\end{document}